\documentclass[11pt]{article}
\usepackage{amsmath,amssymb}

\newtheorem{propo}{{\bf Proposition}}[section]
\newtheorem{coro}[propo]{{\bf Corollary}}
\newtheorem{lemma}[propo]{{\bf Lemma}} \newtheorem{theor}[propo]{{\bf
Theorem}} \newtheorem{ex}{{\sc Example}}[section]

\newenvironment{proof}{{\bf Proof.}}{$\Box$}

\begin{document}

\vspace*{1.0in}

\begin{center} COMPLEMENTS OF INTERVALS AND PREFRATTINI SUBALGEBRAS OF SOLVABLE LIE ALGEBRAS  
\end{center}
\bigskip

\begin{center} DAVID A. TOWERS 
\end{center}
\bigskip

\begin{center} Department of Mathematics and Statistics

Lancaster University

Lancaster LA1 4YF

England

d.towers@lancaster.ac.uk 
\end{center}
\bigskip

\begin{abstract} In this paper we study a Lie-theoretic analogue of a generalisation of the prefrattini subgroups introduced by W. Gasch\"utz. The approach follows that of P. Hauck and H. Kurtzweil for groups, by first considering complements in subalgebra intervals. Conjugacy of these subalgebras is established for a large class of solvable lie algebras. 
\end{abstract}

\section{Complements of subalgebra intervals}
Throughout, $L$ will denote a solvable Lie algebra over a field $F$. For a subalgebra $U$ of $L$ we denote by $[U:L]$ the set of all subalgebras $S$ of $L$ with $U \subseteq S \subseteq L$. We say that $[U:L]$ is {\em complemented} if, for any $S \in [U:L]$ there is a $T \in [U:L]$ such that $S \cap T = U$ and $<S,T> = L$. Our objective is to study the set
\[
\Omega(U,L) = \{S \in [U:L] : [S:L] \hbox{ is complemented}\};
\] 
in particular, to show that, for a large class of solvable Lie algebras $L$, the minimal elements of this set, $\Omega(U,L)_{min}$, are conjugate in $L$. The development initially follows closely that of \cite{hk}. 
\par

We denote by $[U:L]_{max}$ the set of maximal subalgebras in $[U:L]$; that is, the set of maximal subalgebras of $L$ containing $U$. If $L = A + B$ where $A$ and $B$ are subalgebras of $L$ and $A \cap B = 0$ we will write $L = A \oplus B$. 
\bigskip

\begin{lemma}\label{l:one} If $S \in \Omega(U,L)$, $S \neq L$ then $S = \bigcap \{M : M \in [S:L]_{max} \}$.
\end{lemma}
\begin{proof} Put $T = \bigcap \{M : M \in [S:L]_{max} \}$. Then $[S:L]$ is complemented, since $S \in \Omega(U,L)$, and so $T$ has a complement $C$ in $[S:L]$. If $C \neq L$ then $C \subseteq M$ for some $M \in [S:L]_{max}$. But then $<T,C>  = M$, contradicting the fact that $C$ is a complement of $T$ in $[S:L]$. Hence $C = L$ and $S = T \cap C = T \cap L = T$, as required.
\end{proof}
\bigskip

The Frattini subalgebra of $L$, $\phi(L)$, is the intersection of the maximal subalgebras of $L$. When $L$ is solvable this is always an ideal of $L$, by \cite[Lemma 3.4]{bg}. Extending this notion slightly we put $\phi(S,L) = \bigcap \{M : M \in [S:L]_{max} \}$; clearly, $\phi(0,L) = \phi(L)$. The above lemma shows that $\phi(U,L) \subseteq S$ for all $S \in \Omega(U,L)$.

\begin{lemma}\label{l:two} If $I$ is an ideal of $L$ and $S \in \Omega(U,L)$, then $S + I \in \Omega(U,L)$.
\end{lemma}
\begin{proof} Let $B \in [S+I:L] \subseteq [S:L]$. Since $S \in \Omega(U,L)$, $B$ has a complement $D$ in $[S:L]$; that is $B \cap D = S$ and $<B,D> = L$. Put $C = D+I$. Then $<B,C> = L$ and $B \cap C = B \cap (D+I) = B \cap D + I = S + I$, whence $C$ is a complement for $B$ in $[S+I:L]$ and $S+I \in \Omega(U,L)$. 
\end{proof}

\begin{lemma}\label{l:three} Let $A$ be a minimal ideal of $L$ and let $M$ be a complement of $A$ in $L$ containing $U$. Then $\Omega(U,M) = \{ S \in \Omega(U,L) : S \subseteq M \}$. In particular $\Omega(U,M)_{min} = \{ S \in \Omega(U,L)_{min} : S \subseteq M \}$.
\end{lemma}
\begin{proof} Note that since $L$ is solvable, $M$ is a maximal subalgebra of $L$ and $L = A \oplus M$. Suppose first that $S \in \Omega(U,L)$ with $S \subseteq M$. Then $S + A \in \Omega(U,L)$ by Lemma \ref{l:two}. The interval $[S:M]$ is lattice isomorphic to $[S+I:L]$ and so is complemented. Hence $S \in \Omega(U,M)$.
\par

Conversely, let $S \in \Omega(U,M)$. Then $[S:M]$ is complemented. We need to show that $S \in \Omega(U,L)$; that is, that $[S:L]$ is complemented. Let $B \in [S:L]$. Then $B \cap M \in [S:M]$, so there is a subalgebra $D \in [S:M]$ such that $<B \cap M, D> = M$ and $B \cap D = B \cap M \cap D = S$.
\par

If $B \not \subseteq M$ then $M$ is a proper subalgebra of $<B,D>$. But $M$ is a maximal subalgebra of $L$, and so $<B,D> = L$ and $D$ is a complement of $B$ in $[S:L]$. Hence $[S:L]$ is complemented.
\par

If $B \subseteq M$, put $C = D + A$. Then 
$$L = A \oplus M \subseteq \hspace{.1cm} <B,A> + <B,D> \hspace{.1cm} \subseteq \hspace{.1cm} <B, D+A> \hspace{.1cm} = \hspace{.1cm} <B,C>,$$ 
so $<B, D+A> = L$. Also 
$$B \cap C = B \cap (D+A) = B \cap M \cap (D+A) = B \cap (D + M \cap A) = B \cap D = S,$$ yielding that $C$ is a complement of $B$ in $[S:L]$ and $[S:L]$ is complemented. 
\end{proof}

\begin{lemma}\label{l:four} Let $A$ be a minimal ideal of $L$ and let $S \in \Omega(U,L)_{min}$ with $A \not \subseteq S$. Then there is an $M \in [S:L]_{max}$ such that $A \not \subseteq M$.
\end{lemma}
\begin{proof} This follows easily from Lemma \ref{l:one}.
\end{proof}

\begin{lemma}\label{l:five} Let $A$ be a minimal ideal of $L$. Then the following are equivalent:
\begin{itemize}
\item[(i)] $A \not \subseteq S$ for some $S \in \Omega(U,L)_{min}$;
\item[(ii)] $A \not \subseteq M$ for some $M \in [U:L]_{max}$; and
\item[(iii)] for every $S \in \Omega(U,L)_{min}$ there is a complement of $A$ in $L$ containing $S$.
\end{itemize}
\end{lemma}
\begin{proof} $(i) \Rightarrow (ii)$: This follows from Lemma \ref{l:four}.

\noindent $(ii) \Rightarrow (iii)$: Suppose that $A \not \subseteq M$ for some $M \in [U:L]_{max}$. Then $L = A \oplus M$. Let $S \in \Omega(U,L)_{min}$.  
\par

Suppose first that $A \subseteq S$. Then $S = A \oplus M \cap S$ and $M \cap S \cong S/A$, so the interval $[S:L]$ is lattice isomorphic to $[M \cap S:M]$. It follows that $M \cap S \in \Omega(U,M)$. But Lemma \ref{l:three} now gives that $M \cap S \in \Omega(U,L)$, contradicting the minimality of $S$.
\par

Hence $A \not \subseteq S$ and Lemma \ref{l:four} gives a complement of $A$ containing $S$. 

\noindent $(iii) \Rightarrow (i)$: This is trivial.
\end{proof}

\begin{lemma}\label{l:six} If $A$ is an ideal of $L$ and $S \in \Omega(U,L)_{min}$ then $S + A \in \Omega(U+A,L)_{min}$.
\end{lemma}
\begin{proof} It suffices to show that $(S+A)/A \in \Omega((U+A)/A,L/A)_{min}$ and so we may suppose that $A$ is a minimal ideal of $L$. The result is clear if $A \subseteq S$, since then $U+A \subseteq S$. So suppose that $A \not \subseteq S$.
\par

Then there is a complement $M$ of $A$ in $L$ containing $S$, by Lemma \ref{l:five}, and $L = A \oplus M$. Moreover, $S + A \in \Omega(U+A,L)$. Choose $C \in \Omega(U+A,L)_{min}$ such that $C \subseteq S+A$. Then $U \subseteq M \cap C \subseteq S \subseteq M$ and the interval $[M \cap C:M]$ is lattice isomorphic to $[C:L]$. It follows that $M \cap C \in \Omega(U,M)$ and so $M \cap C \in \Omega(U,L)$, by Lemma \ref{l:three}. But $S \in \Omega(U,L)_{min}$, which yields that $M \cap C = S$; that is, $C = S + A$.
\end{proof}
\bigskip

At this point the theory starts to diverge from that for groups. We say that $L$ is {\em completely solvable} if $L^2$ is nilpotent. For these algebras $\Omega(U,L)_{min}$ takes on a particularly simple form.

\begin{theor}\label{t:ellsquared} Let $L$ be completely solvable and let $U$ be a subalgebra of $L$. Then $\Omega(U,L)_{min} = \{\phi(U,L)\}$. In particular, if $U = 0$ then $\Omega(U,L)_{min} = \{\phi(L)\}$.
\end{theor}
\begin{proof} Put $B = \Omega(U,L)_{min}$, $C = \phi(U,L)$. Then $\phi(U,L) \subseteq B$ and so $C \subseteq B$, by Lemma \ref{l:one}. We now use induction on the dimension of $L$. Suppose first that there is a minimal ideal $A$ of $L$ with $A \subseteq C$. Then $B/A \in \Omega((U+A)/A,L/A)_{min}$, by Lemma \ref{l:six}, and so $B/A = \phi((U+A)/A,L/A)$, by the inductive hypothesis. From this it is clear that $B = C$.
\par

So suppose now that no such minimal ideal exists. Then $L$ is $\phi$-free and so $L$ is complemented, by \cite[Theorem 1]{comp}. Thus there is a subalgebra $V$ such that $\langle C,V \rangle = L$ and $C \cap V = 0$. It follows that $\langle C, U+V \rangle = L$ and $C \cap (U+V) = U + C \cap V = U$, whence $C \in [U:L]$ and $[C:L]$ is complemented. Thus $C \in \Omega(U,L)$ and the minimality of $B$ yields that $B = C$.
\end{proof}
\bigskip

If $L$ is not completely solvable then $\Omega(U,L)_{min}$ can contain more than one element as we shall see in the next section. However, we do have a conjugacy result in some cases. First we need to consider inner automorphisms of $L$. Let $x \in L$ and let ad\,$x$ be the corresponding inner derivation of $L$. If $F$ has characteristic zero, suppose that (ad\,$x)^n = 0$ for some $n$; if $F$ has characteristic $p$, suppose that $x \in I$ where $I$ is a nilpotent ideal of $L$ of class less than $p$. Put
\[
\hbox{exp(ad\,}x) = \sum_{r=0}^{\infty} \frac{1}{r!}(\hbox{ad\,}x)^r.
\]
Then exp(ad\,$x)$ is an automorphism of $L$.
\par

If $B$ is a subalgebra of $L$, the {\em centraliser} of $B$ in $L$ is $C_L(B) = \{ x \in L : [x,B] = 0 \}$. We define the {\em nilpotent residual} to be $L^{\infty} = \bigcap_{i=1}^{\infty} L^i$, where $L^i$ are the terms of the lower central series for $L$. Then we have conjugacy for the following metanilpotent Lie algebras.

\begin{theor}\label{t:omegaminconj} Suppose that $L$ is a solvable Lie algebra over a field $F$ of characteristic $p$, and suppose further that $L^{\infty}$ has nilpotency class less than $p$. Let $U$ be a subalgebra of $L$. Then the elements of $\Omega(U,L)_{min}$ are conjugate under ${\mathcal I}(L:L^{\infty})$.
\end{theor}
\begin{proof} We use induction on the dimension of $L$. It is clearly true if $L$ has dimension one, so suppose it holds for such algebras with dimension smaller than that of $L$. We can assume that $L^{\infty} \neq 0$. Let $S_1, S_2 \in \Omega(U,L)_{min}$ and let $A$ be a minimal ideal of $L$ with $A \subseteq L^{\infty}$. Then $(S_1+A)/A, (S_2+A)/A \in \Omega((U+A)/A,L/A)_{min}$, by Lemma \ref{l:six}, and so $(S_1+A)/A$ and $(S_2+A)/A$ are conjugate under ${\mathcal I}(L/A:L^{\infty}/A)$, by the inductive hypothesis.
\par

If $A \subseteq S_1$ then $A \subseteq S_2$, by Lemma \ref{l:five}, and there is an $x \in L^{\infty}$ such that $S_1$ exp(ad\,$x) = S_2$; that is, $S_1$ and $S_2$ are conjugate under ${\mathcal I}(L:L^{\infty})$.
\par

So suppose that $A \not \subseteq S_1$. Then there are complements $M_1$ and $M_2$ of $A$ in $L$ with $S_1 \subseteq M_1$ and $S_2 \subseteq M_2$, by Lemma \ref{l:five}. Put $C = C_{M_1}(A)$, which is an ideal of $L$. If $C = 0$ then $C_L(A) = A$ and there is $a \in A$ such that $M_2$ exp(ad\,$a) = M_1$, by \cite[Theorem 1.1]{bn}, whence $S_2$ exp(ad\,$a) \subseteq M_2$ exp(ad\,$a) = M_1$.
\par

If $C \neq 0$, then $(S_1+C)/C$ and $(S_2+C)/C$ are conjugate under ${\mathcal I}(L/C:(L^{\infty} + C)/C)$, by the inductive hypothesis. It follows that there is an $x \in L^{\infty}$ such that $S_2$ exp(ad\,$x + C) \subseteq S_1 + C$ exp(ad\,$a) \subseteq M_1$, which gives $S_2$ exp(ad\,$x) \subseteq M_1$. Now $L = A \oplus M_1$, so $L^{\infty} \subseteq A \oplus M_1^{\infty}$. Moreover, $[A,L^{\infty}] = 0$ since $L^{\infty}$ is nilpotent, so $M_1^{\infty}$ is an ideal of $L$. Put $x = a + b$, where $a \in A$, $b \in M_1^{\infty}$. Then, for each $s_2 \in S_2$, we have $s_2 + s_2$ ad\,$x + \ldots + s_2$ (ad\,$x)^n \in M_1$, which gives $s_2 + s_2$ ad\,$a \in M_1$. Thus, again we have that $S_2$ exp(ad\,$a) \subseteq M_1$ for some $a \in A$.
\par

So $S_1, S_2$ exp(ad\,$a) \subseteq M_1$ for some $a \in A$. Now $U \subseteq S_1 \subseteq M_1$ and $U$ exp(ad\,$a) \subseteq S_2$ exp(ad\,$a) \subseteq M_1$, so, for each $u \in U$, $u + [a,u] \in M_1$ which gives $[a,u] \in A \cap M_1 = 0$; that is, $a \in C_L(U)$ and $U$ exp(ad\,$a) = U$. Thus $S_2$ exp(ad\,$a) \in \Omega(U$ exp(ad\,$a), L)_{min} = \Omega(U,L)_{min}$. But now Lemma \ref{l:three} yields that $S_1, S_2$ exp(ad\,$a) \in \Omega(U,M_1)_{min}$ and the required conjugacy of $S_1$ and $S_2$ follows from the inductive hypothesis. 
\end{proof}

\section{U-prefrattini subalgebras}
Let 
\begin{equation} 0 = A_0 \subset A_1 \subset \ldots \subset A_n = L  
\end{equation} be a fixed chief series for $L$. We say that $A_i/A_{i-1}$ is a {\em Frattini} chief factor if $A_i/A_{i-1} \subseteq \phi(L/A_{i-1})$; it is {\em complemented} if there is a maximal subalgebra $M$ of $L$ such that $L = A_i + M$ and $A_i \cap M = A_{i-1}$. When $L$ is solvable it is easy to see that a chief factor is Frattini if and only if it is not complemented. This can be generalised as follows.  

The factor algebra $A_i/A_{i-1}$ is called a {\em $U$-Frattini} chief factor if
\[ A_i \subseteq \phi(U+A_{i-1},L) \hbox{ or if } U + A_{i-1} = L.
\]
If $A_i/A_{i-1}$ is not a $U$-Frattini chief factor there is an $M \in [U+A_{i-1}:L]_{max}$ for which $A_i \not \subseteq M$; that is, $M$ is a complement of the chief factor $A_i/A_{i-1}$.  We have a sharpened form of the Jordan-H\"older Theorem in which $U$-Frattini chief factors correspond. First we need a lemma.

\begin{lemma}\label{l:corr} Let $A_1$, $A_2$ be distinct minimal ideals of the solvable Lie algebra $L$. Then there is a bijection 
\[ \theta: \{ A_1, (A_1 + A_2)/A_1 \} \rightarrow \{ A_2, (A_1 + A_2)/A_2 \}
\]
such that corresponding chief factors have the same dimension and $U$-Frattini chief factors correspond to one another.
\end{lemma}
\begin{proof} Clearly we can assume that $U \neq L$. Put $A = A_1 \oplus A_2$. Suppose first that $A_1$ is a $U$-Frattini chief factor. Then $A_1 \subseteq \phi(U,L)$. Thus $A \subseteq \phi(U + A_2, L)$ and $A/A_2$ is a $U$-Frattini chief factor. If $A/A_1$ is also a $U$-Frattini chief factor, then $A \subseteq \phi(U + A_1, L)$, which yields that $A \subseteq \phi(U,L)$, and all four factors are $U$-Frattini. In this case we can choose $\theta$ so that $\theta(A_1) = A/A_2$ and $\theta(A/A_1) = A_2$. If $A/A_1$ is not a $U$-Frattini chief factor, then nor is $A_2$, by the same argument as above, and so the same choice of $\theta$ suffices; likewise if none of the factors are $U$-Frattini chief factors.
\par

The remaining case is where $A_1$ and $A_2$ are not $U$-Frattini chief factors but $A/A_2$ is. Then $A_1 \not \subseteq \phi(U,L)$, $A_2 \not \subseteq \phi(U,L)$ and either $A \subseteq \phi(U + A_2, L)$ or $U + A_2 = L$. Thus there exists $M \in [U,L]_{max}$ such that $A_1 \not \subseteq M$, giving $L = A_1 \oplus M$. Put $A_3 = M \cap A$. Then $A_3 \oplus A_1 = M \cap A \oplus A_1 = (M + A_1) \cap A = A$, and so $A_3 \cong A/A_1 \cong A_2$. If $A_3 = A_2$ then $U + A_2 \subseteq M$ which gives $A \subseteq M$: a contradiction. Hence $A_3 \neq A_2$, $A = A_3 \oplus A_2$ and $A_3 \cong A/A_2 \cong A_1$. It follows that all of the chief factors have the same dimension.
\par

If $U + A_1 = L$, then $A/A_1$ is a $U$-Frattini chief factor, so we can choose $\theta$ so that $\theta(A_1) = A_2$ and $\theta(A/A_1) = A/A_2$. If $U + A_1 \neq L$, let $N \in [U + A_1, L]_{max}$. If $A_2 \not \subseteq N$ then $L = A_2 \oplus N$ and $N \cap A = A_1$. But $A \subseteq \phi(U + A_2, L)$ implies that $A \subseteq \phi(U,L) + A_2$, whence $A + A_2 + \phi(U,L) \cap A$. It follows that $\phi(U,L) \cap A \subseteq N \cap A = A_1$, giving $\phi(U,L) \cap A = A_1$. But now $A_1 \subseteq \phi(U,L)$: a contradiction. We must, therefore, have $A_2 \subseteq N$ and so $A \subseteq N$. Thus $A \subseteq \phi(U + A_1, L)$; that is, $A/A_1$ is a $U$-Frattini chief factor. In this case we can again choose $\theta$ so that $\theta(A_1) = A_2$ and $\theta(A/A_1) = A/A_2$.    
\end{proof}

\begin{theor}\label{t:jordan} Let
\[ 0 < A_1 < \ldots < A_n = L  \hspace{1in} (1)
\]
\[ 0 < B_1 < \ldots < B_n = L \hspace{1in} (2)
\]
be chief series for the solvable Lie algebra $L$. Then there is a bijection between the chief factors of these two series such that corresponding factors have the same dimension and such that the $U$-Frattini chief factors in the two series correspond. 
\end{theor}
\begin{proof} These two series have the same length by a version of the Jordan H\"older Theorem. We induction on $n$. The result is clearly true if $n = 1$. So let $n > 1$ and suppose that the result holds for all solvable Lie algebras with chief series of length $\leq n-1$. If $A_1 = B_1$, then applying the inductive hypothesis to $L/A_1$ gives a suitable bijection between the factors above $A_1$, and then we can map $A_1$ to $B_1$ and we have the result.
\par

So suppose that $A_1$ and $B_1$ are distinct and put $A = A_1 \oplus B_1$. Then $A/A_1$ and $A/B_1$ are chief factors of $L$ and there are chief series of the form
\[ 0 < A_1 < A < C_3 < \ldots < C_n = L  \hspace{1in} (3)
\]
\[ 0 < B_1 < A < C_3 < \ldots < C_n = L \hspace{1in} (4)
\]
Define an equivalence relation on the chief series of $L$ by saying that two such series are equivalent if there is a bijection between their chief factors satisfying the requirements of the theorem. Since series $(1)$ and $(3)$ have a minimal ideal in common, they are equivalent. Similarly, sereis $(2)$ and $(4)$ are equivalent. Moreover, since series $(3)$ and $(4)$ coincide above $A$ they are also equivalent, by Lemma \ref{l:corr}. Hence the series $(1)$ and $(2)$ are equivalent, as required.
\end{proof}
\bigskip

We define the set $\mathcal{I}$ by $i \in \mathcal{I}$ if and only if $A_i/A_{i-1}$ is not a $U$-Frattini chief factor of $L$. For each $i \in \mathcal{I}$ put
\[ \mathcal{M}_i = \{ M \in [U + A_{i-1}, L]_{max} \colon A_i \not \subseteq M\}.
\]
Then $B$ is a {\em $U$-prefrattini} subalgebra of $L$ if 
\[ B = \bigcap_{i \in \mathcal{I}} M_i \hbox{ for some } M_i \in \mathcal{M}_i.
\]
If $U = 0$ we will refer to $B$ simply as a {\em prefrattini} subalgebra of $L$.
\par

The subalgebra $B$ {\em avoids} $A_i/A_{i-1}$ if $B \cap A_i = B \cap A_{i-1}$; likewise, $B$ {\em covers} $A_i/A_{i-1}$ if $B + A_i = B + A_{i-1}$. Then we have the following important property of $U$-prefrattini subalgebras of $L$. 

\begin{lemma}\label{l:cover} If $B$ is a $U$-prefrattini subalgebra of $L$ then it covers all $U$-Frattini chief factors of $L$ in $(1)$ and avoids the rest.
\end{lemma}
\begin{proof} Let $B$ be a $U$-prefrattini subalgebra of $L$ and let $A_i/A_{i-1}$ be a chief factor of $L$. If it is a $U$-Frattini chief factor then either $A_i \subseteq \phi(U+A_{i-1},L)$ or else $U+A_{i-1} = L$. In the former case, every maximal subalgebra of $L$ that contains $U+A_{i-1}$ also contains $A_i$, and so $A_i \subseteq B$. In either case, therefore, $B$ covers $A_i/A_{i-1}$. If it is not a $U$-Frattini chief factor we have $B \subseteq M_i$ where $L = A_i + M_i$ and $A_i \cap M_i = A_{i-1}$. Hence $B \cap A_i = B \cap M_i \cap A_i = B \cap A_i-1 \subseteq B \cap A_i$, and so $B$ avoids $A_i/A_{i-1}$. 
\end{proof} 
\bigskip

The next four results are dedicated to showing how the $U$-prefrattini subalgebras relate to the material in the previous section. The first lemma is helpful when trying to calculate $U$-prefrattini subalgebras.

\begin{lemma}\label{l:dim} Let $B$ be a $U$-prefrattini subalgebra of $L$. Then 
\[ \dim B = \sum_{i \notin {\mathcal I}}(\dim A_i - \dim A_{i-1});
\]
in particular, all $U$-prefrattini subalgebras of $L$ have the same dimension.
\end{lemma}
\begin{proof} We use induction on $\dim L$. The result is clear if $L$ is abelian, so suppose it holds for Lie algebras of smaller dimension than $L$. It is easy to check that $(B + A_1)/A_1$ is a $((U + A_1)/A_1)$-prefrattini subalgebra of $L/A_1$ and so 
\[ \dim \left( \frac{B+A_1}{A_1} \right) = \sum_{i \in I, i \neq 1}(\dim A_i - \dim A_{i-1}),
\]
by the inductive hypothesis. If $A_1/A_0$ is a $U$-Frattini chief factor of $L$, then $B$ covers $A_1/A_0$, whence $B = B+A_1$ and 
\[ \dim B = \dim A_1 + \dim \left( \frac{B+A_1}{A_1} \right) = \sum_{i \in I}(\dim A_i - \dim A_{i-1}).
\]
If $A_1/A_0$ is not a $U$-Frattini chief factor of $L$, then $B$ avoids $A_1/A_0$, whence $B \cap A_1 = 0$ and 
\[ \dim B = \dim \left( \frac{B+A_1}{A_1} \right) = \sum_{i \in I}(\dim A_i - \dim A_{i-1}).
\]
\end{proof}
\bigskip

Let $\Pi(U,L)$ be the set of $U$-prefrattini subalgebras of $L$.

\begin{lemma}\label{l:preomega} $\Pi(U,L) \subseteq \Omega(U,L)$.
\end{lemma}
\begin{proof} (i) We use induction on $\dim L$. The result is clear if $L$ is abelian, so suppose it holds for Lie algebras of dimension less than that of $L$. Let $B \in \Pi(U,L)$. Then  
\[ \frac{B + A_1}{A_1} \in \Pi \left(\frac{U+A_1}{A_1},\frac{L}{A_1}\right) \subseteq \Omega \left(\frac{U+A_1}{A_1},\frac{L}{A_1} \right), 
\]
whence $B + A_1 \in \Omega(U,L)$. If $A_1 \subseteq B$ we have $B \in \Omega(U,L)$. So suppose that $A_1 \not \subseteq B$. Then $B$ does not cover $A_1/A_0$, so $A_1/A_0$ is not a $U$-Frattini chief factor of $L$. It follows that $1 \in {\mathcal I}$, and so there is a maximal subalgebra $M$ of $L$ with $B \subseteq M$ and $A_1 \not \subseteq M$. But now $L = A_1 \oplus M$ and the intervals $[B+A_1:L]$ and $[B:M]$ are lattice isomorphic, which yields that $[B:M]$ is complemented. It follows from Lemma \ref{l:three} that $B \in \Omega(U,L)$ again. 
\end{proof}

\begin{lemma}\label{l:omegapre} $\Omega(U,L)_{min} \subseteq \Pi(U,L)$.
\end{lemma}
\begin{proof} Let $B \in \Omega(U,L)_{min}$ and let $A_i/A_{i-1}$ be a chief factor of $L$. By Lemma \ref{l:six}, 
\[ \left( \frac{B + A_{i-1}}{A_{i-1}}\right) \in \Omega \left( \frac{U+A_{i-1}}{A_{i-1}},\frac{L}{A_{i-1}} \right)_{min}.
\]
We now apply Lemma \ref{l:five} to the minimal ideal $A_i/A_{i-1}$ of $L/A_{i-1}$. If $A_i/A_{i-1}$ is a $U$-Frattini chief factor then it doesn't have a complement in $L/A_{i-1}$ and Lemma \ref{l:five} gives that $A_i \subseteq B+A_{i-1}$, whence $A_i + B = A_{i-1} + B$ and $B$ covers $A_i/A_{i-1}$.
\par

If $A_i/A_{i-1}$ is not a $U$-Frattini chief factor then it has a complement $M_i/A_{i-1}$ in $L/A_{i-1}$ and Lemma \ref{l:five} gives that it has such a complement containing $(B+A_{i-1})/A_{i-1}$; that is $L = M_i + A_i$, $M_i \cap A_i = A_{i-1}$ and $B + A_{i-1} \subseteq M_i$. But now $B \cap A_i \subseteq B \cap A_i + A_{i-1}  = (B + A_{i-1}) \cap A_i \subseteq M_i \cap A_i = A_{i-1}$. It follows that $B \cap A_i = B \cap A_{i-1}$ and $B$ avoids $A_i/A_{i-1}$. Clearly $M_i \in {\mathcal M_i}$ and $B \subseteq C = \bigcap_{i \in {\mathcal I}} M_i \in \Pi(U,L)$. But $B$ covers or avoids the same chief factors of $(1)$ as $C$, so the proof of Lemma \ref{l:dim} shows that $\dim B = \dim C$. It follows that $B = C \in \Pi(U,L)$.
\end{proof}
\bigskip

Putting the previous three Lemmas together yields the following result.

\begin{theor}\label{t:prefrat} $\Omega(U,L)_{min} = \Pi(U,L)$.
\end{theor}
\bigskip

Notice that, in particular, the above result shows that the definition of $U$-prefrattini subalgebras does not depend on the choice of chief series.

\begin{coro}\label{c:six} If $A$ is an ideal of $L$ and $S \in \Pi(U,L)$ then $(S + A)/A \in \Pi((U+A)/A,L/A)$.
\end{coro}
\begin{proof} This follows from Theorem \ref{t:prefrat} and Lemma \ref{l:six}.
\end{proof}

\begin{coro}\label{c:fratsub} For every solvable Lie algebra $L$, 
\[ \phi(U,L) = \bigcap_{B \in \Pi(U,L)} B.
\]
\end{coro}
\begin{proof} Put $P = \bigcap_{B \in \Pi(U,L)} B$. Then $\phi(U,L) \subseteq P$, by Theorem \ref{t:prefrat} and Lemma \ref{l:one}. Let $M \in [U, L]_{max}$. There is an $i$ such that $A_{i-1} \subseteq M$ but $A_i \not \subseteq M$ ($1 \leq i \leq n$). Then $A_i/A_{i-1}$ is not a $U$-Frattini chief factor of $L$, so $i \in {\mathcal I}$  and $M \in {\mathcal M_i}$. Thus there is $B \in \Pi(U,L)$ such that $B \subseteq M$, whence $P \subseteq M$. Hence $P \subseteq \phi(U,L)$. 
\end{proof}

\begin{coro}\label{c:ellsquared} Let $L$ be completely solvable and let $U$ be a subalgebra of $L$. Then $\Pi(U,L) = \{\phi(U,L)\}$. In particular, $\Pi(0,L) = \{\phi(L)\}$.
\end{coro}
\begin{proof} This follows from Theorem \ref{t:prefrat} and Theorem \ref{t:ellsquared}.
\end{proof}

\begin{coro}\label{c:omegaminconj} Suppose that $L$ is a solvable Lie algebra over a field $F$ of characteristic $p$, and suppose further that $L^{\infty}$ has nilpotency class less than $p$. Let $U$ be a subalgebra of $L$. Then the elements of $\Pi(U,L)$ are conjugate under ${\mathcal I}(L:L^{\infty})$.
\end{coro}
\begin{proof} This follows from Theorem \ref{t:prefrat} and Theorem \ref{t:omegaminconj}.
\end{proof}
\bigskip

If $L^2$ is not nilpotent then $\Pi(U,L)$ can contain more than one element, as the following example shows.

\begin{ex} Let $F$ be a field of characteristic $p$ (perfect if $p = 2$), and $L = (\oplus_{i=0}^{p-1} Fe_i) \oplus Fc \oplus Fs \oplus Fx$ with $[e_i,c] = e_i$, $[e_i,s] = e_{i+1}$ for $i = 0, \ldots, p-2$, $[e_{p-1},s] = 0$, $[e_i,x] = ie_{i-1}$ for $i = 0, \ldots,p-1$ and $e_{-1} = 0$, $[s,x] = c$, and all other products zero.
\end{ex}

Put $A_0 = 0$, $A_1 = \oplus_{i=0}^{p-1} Fe_i$, $A_2 = A_1 \oplus Fc$, $A_3 = A_2 \oplus Fs$, $A_4 = L$. Then
\[ 0 = A_0 \subset A_1 \subset A_2 \subset A_3 \subset A_4 = L
\] 
is a chief series for $L$ in which $A_2/A_1$ is the only Frattini chief factor. It is, therefore, straightforward to see that the prefrattini subalgebras of $L$ are the one-dimensional subalgebras $F(\alpha c + a)$ where $a \in A_1 = L^{\infty}$, $\alpha \in F$. 
\par

Note that these are all conjugate under inner automorphisms of the form $1 +$ ad\,$a$. This is not always the case, however. For, if $B$ is a faithful completely reducible $L$-module and we form $X = B \oplus L$, where $B^2 = 0$ and $L$ acts on $B$ under the given $L$-module action, then the prefrattini subalgebras are still of the form $F(\alpha c + a)$ where $a \in A_1$. However, $B$ is the unique minimal ideal of $L$ and these subalgebras are not conjugate under inner automorphisms of the form $1 +$ ad\,$b$, $b \in B$. Since $B$ is the nilradical of $X$, defining other inner automorphisms is problematic. Note that $X^{\infty} = B + A_1$ which is not nilpotent.

\end{document}